\documentclass[12pt]{article}

\usepackage[margin=2.54cm]{geometry}
\usepackage{setspace}
\usepackage{amsmath}
\usepackage{amssymb}
\usepackage{amsthm}
\usepackage{color}

\DeclareMathOperator{\st}{s.t.}
\DeclareMathOperator{\diag}{diag}

\DeclareMathOperator{\rank}{rank}

\newtheorem{theorem}{Theorem}
\newtheorem{corollary}{Corollary}
\newtheorem{lemma}{Lemma}
\newtheorem{proposition}{Proposition}
\newtheorem{assumption}{Assumption}

\newcommand{\RR}{{\mathbb{R}}}
\newcommand{\SSS}{{\mathbb{S}}}

\title{Exact Semidefinite Formulations for a \\ Class of 
(Random and Non-Random) \\ Nonconvex Quadratic Programs}

\author{%
Samuel Burer%
\thanks{Department of Management Sciences, University of Iowa,
Iowa City, IA, 52242-1994, USA. Email: {\tt samuel-burer@uiowa.edu}.}%
\and
Yinyu Ye%
\thanks{Department of Management Science \& Engineering
and Institute of Computational \& Mathematical Engineering, Stanford University,
Stanford, CA, 94305-4121, USA. Email: {\tt yinyu-ye@stanford.edu}.}
}

\date{February 7, 2018 \\ Revised: February 14, 2018 \\ Revised: November 7, 2018}

\begin{document}

\maketitle

\begin{abstract}
We study a class of quadratically constrained quadratic programs
(QCQPs), called {\em diagonal QCQPs\/}, which contain no off-diagonal
terms $x_j x_k$ for $j \ne k$, and we provide a sufficient condition
on the problem data guaranteeing that the basic Shor semidefinite
relaxation is exact. Our condition complements and refines those already
present in the literature and can be checked in polynomial time. We
then extend our analysis from diagonal QCQPs to general QCQPs, i.e.,
ones with no particular structure. By reformulating a general QCQP
into diagonal form, we establish new, polynomial-time-checkable sufficient conditions for the
semidefinite relaxations of general QCQPs to be exact. Finally, these
ideas are extended to show that a class of random general QCQPs has
exact semidefinite relaxations with high probability as long as the
number of constraints grows no faster than a fixed polynomial in the
number of variables. To the best of our knowledge, this is the first
result establishing the exactness of the semidefinite relaxation for
random general QCQPs.

\mbox{}

\end{abstract}

\begin{onehalfspace}

\section{Introduction}

We study {\em quadratically constrained quadratic programming\/} (QCQP),
i.e., the minimization of a nonconvex quadratic objective over the intersection
of nonconvex quadratic constraints:
\begin{align}
\min \ \ & x^T C x + 2 c^T x \label{equ:qcqp} \\
\st \ \ & x^T A_i x + 2 a_i^T x \le b_i \ \ \forall \ i = 1, \ldots, m. \nonumber
\end{align}
The variable is $x \in \RR^n$ and the data consist of the symmetric
matrices $\{ C, A_i \}$ and column vectors $\{ c, a_i \}$. QCQPs subsume
a wide variety of NP-hard optimization problems, and hence a reasonable
approach is to approximate them via tractable classes of optimization
problems.

{\em Semidefinite programming\/} (SDP) is one of the most frequently
used tools for approximating QCQPs in polynomial time \cite{Todd_2001, Anjos-Lasserre_2012}. The standard
approach constructs an SDP relaxation of (\ref{equ:qcqp}) by replacing
the rank-1 matrix inequality ${1 \choose x}{1 \choose x}^T \succeq 0$ by
$Y(x,X) \succeq 0$, where
\[
    Y(x,X) := \begin{pmatrix} 1 & x^T \\ x & X \end{pmatrix} \in \SSS^{n+1}
\]
and $\SSS^{n+1}$ denotes the symmetric matrices of size $(n + 1) \times
(n + 1)$. In this paper, we focus on the simplest SDP
relaxation of (\ref{equ:qcqp}), called the Shor relaxation \cite{Shor_1987}:
\begin{align}
\min \ \ & C \bullet X + 2 c^T x \label{equ:psdp} \\
\st \ \ & A_i \bullet X + 2 a_i^T x \le b_i \ \ \forall \ i = 1, \ldots, m \nonumber \\
&Y(x,X) \succeq 0. \nonumber
\end{align}
where $M \bullet N := \mbox{trace}(M^T N)$ is the trace inner product.

\subsection{Rank bounds}

Let $r^*$ be the smallest rank among all optimal solutions $Y^* :=
Y(x^*, X^*)$ of (\ref{equ:psdp}). When $r^* = 1$, the relaxation
(\ref{equ:psdp}) solves (\ref{equ:qcqp}) exactly, and loosely
speaking, $r^*$ is an important measure for understanding the
quality of the SDP relaxation, e.g., a low $r^*$ might allow one to
develop an approximation algorithm for (\ref{equ:qcqp}) by solving
(\ref{equ:psdp}). Furthermore: in many cases the true objective of
interest is to find a low-rank feasible solution of (\ref{equ:psdp})
\cite{Recht-et-al_2010}; and knowing $r^*$, or simply preferring a
smaller rank, can even help with solving (\ref{equ:psdp}) via so-called
{\em low-rank approaches\/} for solving SDPs \cite{Burer-Monteiro_2003}.

We are interested in {\em a priori\/} upper bounds on $r^*$. Pataki
\cite{Pataki_1998}, Barvinok \cite{Barvinok_1995}, and Deza-Laurent
\cite{Deza-Laurent_1997} independently proved that $r^*(r^* + 1)/2 \le
m+1$, or equivalently $r^* \le \lceil \sqrt{2(m+1)} \rceil$.\footnote{In
fact, if the number of inactive linear inequalities at $Y^*$ is known
ahead of time, then this bound can be improved. For example, suppose
(\ref{equ:psdp}) contains the two inequalities $0 \le X_{12} \le 1$.
Then the rank bound can be improved to $\lceil \sqrt{2m} \rceil$ since
both inequalities cannot be active at the same time.} Note that this
result depends neither on $n$ nor on the data of the SDP. In general,
to reduce the bound further, one must exploit the particular structure
and/or data of the instance, and there are many examples in which this
is indeed possible
\cite{Sturm-Zhang_2003,Ye.Zhang.2003a, Sojoudi-Lavaei_2014, Burer_2015}.
For example, one classical result establishes that, if all $C, A_i$ are
positive semidefinite, then $r^* = 1$ is guaranteed \cite{Fujie-Kojima_1997}. 

A recent approach bounds $r^*$ by studying the structure of
the simple, undirected graph $G$ defined by the aggregate nonzero
structure of the matrices
\[
    \tilde A_0 := \begin{pmatrix} 1 & c^T \\ c & C \end{pmatrix}, \
    \tilde A_1 := \begin{pmatrix} 1 & a_1^T \\ a_1 & A_m \end{pmatrix}, \ldots, \
    \tilde A_m := \begin{pmatrix} 1 & a_m^T \\ a_m & A_m \end{pmatrix}.
\]
Specifically, $G := (V, E)$, where $V := \{ 0, 1, \ldots, n \}$ and
\[
    E := \left\{ (j,k) : [\tilde A_i]_{jk} \ne 0 \text{ for some } i \in \{0, 1, \ldots m \} \right\}
\]
Laurent and Varvitsiotis \cite{Laurent-Varvitsiotis_2014} show
in particular that $r^*$ is bounded above by the tree-width
of $G$ plus 1; see \cite{Diestel_2018} for a definition of tree-width. So, for example when $G$ is a tree, $r^*
\le 2$. Similar approaches and extensions can be found in
\cite{Sojoudi-Lavaei_2014,Madani-et-al_2014,Mandani-et-al_2017}. In
fact, \cite{Madani-et-al_2014} proves that any polynomial optimization
problem can be reformulated as a QCQP with a corresponding SDP
relaxation having $r^* \le 2$. This demonstrates that, in a certain
sense, the difference between approximating and solving (\ref{equ:qcqp}) is
precisely the difference between $r^* = 2$ and $r^* = 1$.

In this paper, we study, new sufficient conditions guaranteeing $r^* =
1$, and we
do so in two stages.

First, in Section \ref{sec:rankbound}, we consider a subclass of
QCQPs that we call {\em diagonal QCQPs\/}: each data matrix $C, A_1,
\ldots, A_m$ is diagonal. This means that no cross terms $x_j x_k$
for $j \ne k$ appear in (\ref{equ:qcqp}), and hence each quadratic
function is separable, although the entire problem is not. Under a
linear transformation, this is equivalent to the conditions that all
$C, A_i$ pairwise commute and that all $C, A_i$ share a common basis of
eigenvectors. This subclass is itself NP-hard since it contains, for
example, 0-1 binary integer programs. In addition, in this case, the
aggregate nonzero structure of $G$ is a star, which is a type of tree, with the first $n$ vertices
connected to the $(n+1)$-st vertex, and hence, as discussed above, $r^*
\le 2$ for diagonal QCQPs. A constant approximation algorithm based on
the SDP relaxation was given by \cite{Ye_1999}.

With respect to diagonal QCQPs, our main result provides a sufficient
condition on the data of (\ref{equ:qcqp}) guaranteeing $r^* = 1$.
Independent of the Laurent-Varvitsiotis bound, which is based only
on the graph structure $G$, our approach shows that $r^*$ is bounded
above by $n - f + 1$, where $f$ is a data-dependent integer that can
be computed in a prepocessing step by solving $n$ linear programs
(LPs). Specifically, before solving the relaxation (\ref{equ:psdp}), we
construct and solve $n$ auxiliary LPs using the data of (\ref{equ:qcqp})
to assess the feasibility of $n$ polyhedral systems. The integer $f$ is
the number of those systems, which are feasible, and then we prove
$r^* \le n - f + 1$. Thus the condition $f = n$ implies $r^* = 1$. In
particular, the $j$-th linear system employs the data $C, A_1, \ldots,
A_m$ and $c_j, a_{1j}, \ldots, a_{mj}$ and contains 1 equation,
$m$ inequalities, $n-1$ nonnegative variables, and 2 free variables;
see (\ref{equ:pfeasj}) below. Note that $f$ does not depend on $b$. In
contrast with the Laurent-Varvitsiotis bound, our bound depends both
on the graph structure and the problem data itself. Also, while our
bound $r^* \le n - f + 1$ is not as strong as theirs in general, it
can be stronger in specific cases as we will demonstrate. For example,
we reprove a result from \cite{Sojoudi-Lavaei_2014}, which also
exploits conditions on the data to guarantee $r^* = 1$ for a particular sub-class
of diagonal QCQPs. We also provide an example showing that our
analysis can be stronger than that of \cite{Sojoudi-Lavaei_2014} in
certain cases.

Second, in Section \ref{sec:general}, we study the case of general,
non-diagonal QCQPs by first reducing to the case of diagonal QCQPs.
This is done by a standard introduction of auxiliarly variables that
lifts (\ref{equ:qcqp}) to a higher dimension in which the QCQP is
diagonal.\footnote{Interestingly, compared to \cite{Madani-et-al_2014},
this provides a simple proof that every polynomial optimization problem
has a corresponding SDP relaxation in which $r^* = 2$; see Section \ref{sec:general}.} Then, by
applying the theory for diagonal QCQPs to this lifted QCQP, we obtain
sufficient conditions for the SDP relaxation (\ref{equ:psdp}) of the
original (\ref{equ:qcqp}) to be exact with $r^* = 1$. These sufficient
conditions involve only the eigenvalues of the matrices $C, A_1, \ldots,
A_m$ and do not depend on the vectors $c, a_1, \ldots, a_m, b$.

\subsection{Rank bounds under data randomness}

Our proof techniques in Sections \ref{sec:rankbound} and
\ref{sec:general} reveal an interesting property of the bound $r^* \le
n - f + 1$ mentioned above, namely that it can often be improved by a
simple perturbation of the data of (\ref{equ:qcqp}). In this paper, in addition to
examining data perturbation, we also consider how the rank bound $r^*
\le n - f + 1$ behaves under random-data models. Our interest in this
subject arises from the fact that optimization algorithms have recently
been applied to solve problems for which data are random, often because
data themselves contain randomness in a big-data environment or are
randomly sampled from large populations.

It has been shown that data randomness typically makes algorithms
run faster in the so-called {\em average behavior analysis\/}.
The idea is to obtain rigorous probabilistic bounds on the number
of iterations required by an algorithm to reach some termination
criterion when the algorithm is applied to a random instance of a
problem drawn from some probability distribution. In the case of
the simplex method for LP, average-case analyses have provided some
theoretical justification for the observed practical efficiency of
the method, despite its exponential worst-case bound; see for example
\cite{Adler-Megiddo_1985, Borgwardt_1987, Smale_1983, Todd_1986}.

In the case of interior-point algorithms for LP, a ``high probability"
bound of $O(\sqrt{n}\ln n)$ iterations for termination (independent of
the data size) has been proved using a variety of algorithms applied to
several different probabilistic models. Here, $n$ is the dimension or
number of variables in a standard form problem, and ``high probability"
means the probability of the bound holding goes to 1 as $n\to
\infty$; see, e.g., \cite{Ye_1994, Anstreicher-et-al_1999}. The paper
\cite{Todd-et-al_2001} analyzed a condition number of the constraint
matrix $A$ of dimension $m\times n$ for an interior-point LP algorithm
and showed that, if $A$ is a standard Gaussian matrix, then the expected
condition number equals $O(\min\{m\ln n, n\})$. Consequently, the
algorithm terminates in strongly polynomial time in expectation.

On the other hand, specific recovery problems with random
data/sampling have been proved to be exact via convex optimization
approaches, which include digital communication \cite{So_2010},
sensor-network localization \cite{Shamsi-et-al_2013}, PhaseLift
signal recovery \cite{Candes-et-al_2013}, and max-likelihood angular
synchronization \cite{Bandeira-et-al_2017}; see also the survey paper
\cite{Luo-et-al_2010} and references therein. In these approaches,
the recovery problems are relaxed to semidefinite programs (SDPs),
where each randomly sampled measurement becomes a constraint in the
relaxation. When the number of random constraints or measurements is
sufficiently large---$O(n\ln n)$ relative to the dimension $n$ of the
variable matrix---then the relaxation contains the unique solution
to be recovered. However, these problems can actually be solved by
more efficient, deterministic, targeted sampling using only $O(n)$
measurements.

In Section \ref{sec:random}, for general QCQPs, we give further evidence
to show that a nonconvex optimization problem, for which the data are
random and the number of constraints $m$ grows as a fixed polynomial in
the variable dimension $n$, can be globally solved with high probability
via convex optimization, specifically SDP. The proof is based on the
ideas developed in Sections \ref{sec:rankbound} and \ref{sec:general}.

We mention briefly that our approach of analyzing random
problems, i.e., problems generated from a particular probability
distribution, differs from the smoothed-analysis approaches of
papers such as \cite{Spielman-Teng_2004} for LP and
\cite{Bhojanapalli-et-al_2018} for SDP. Smoothed analysis makes
no distributional assumptions and proves good algorithmic behavior or
good problem characteristics on all problems except a set of measure zero and
hence differs from the techniques herein.

\subsection{Assumptions and basic setup}

We make the following assumptions throughout:

\begin{assumption} \label{ass:feas}
The feasible set of (\ref{equ:qcqp}) is nonempty.
\end{assumption}

\begin{assumption} \label{ass:pd}
There exists $\bar y \in \RR^m$ such that $\bar y \le 0$ and $\sum_{i = 1}^m
\bar y_i A_i \prec 0$.
\end{assumption}

\noindent Assumption \ref{ass:pd} could be equivalently stated with
$\bar y \ge 0$ and $\sum_{i=1}^m \bar y_i A_i \succ 0$. However, this 
form with $\bar y \le 0$ will match the SDP dual (\ref{equ:dsdp}) below. Define
\[
    \bar A := \sum_{i=1}^m \bar y_i A_i, \ \
    \bar a := \sum_{i=1}^m \bar y_i a_i.
\] 
Assumptions \ref{ass:feas}--\ref{ass:pd} together imply that
the feasible set of (\ref{equ:qcqp}) is contained within the
full-dimensional ellipsoid defined by $-x^T \bar A x - 2 \bar a^T x \le
-b^T \bar y$ and hence (\ref{equ:qcqp}) has an optimal solution. This also
implies that the feasible set of (\ref{ass:pd}) is bounded due to its
redundant constraint $-\bar A \bullet X - 2 \bar a^T x \le -b^T \bar y$,
which again uses $\bar A \prec 0$. We also assume:

\begin{assumption} \label{ass:int}
The interior feasible set of (\ref{equ:psdp}) is nonempty.
\end{assumption}

\noindent The dual of (\ref{equ:psdp}) is
\begin{align}
\max \ \ & b^T y - \lambda \label{equ:dsdp} \\
\st \ \ \, & y \le 0, \ Z(\lambda, y) \succeq 0 \nonumber
\end{align}
where
\[
    Z(\lambda, y) := \begin{pmatrix} \lambda & s(y)^T \\ s(y) & S(y) \end{pmatrix}
\]
and
\[
    s(y) := c - \sum_{i=1}^m y_i a_i, \ \ \ \ \ \ \  S(y) := C - \sum_{i=1}^m y_i A_i.
\]
Assumption \ref{ass:pd} also implies that the feasible set of
(\ref{equ:dsdp}) has interior, which together with Assumption
\ref{ass:int} ensures that strong duality holds between (\ref{equ:psdp})
and (\ref{equ:dsdp}), i.e., there exist feasible $Y^* := Y(x^*, x^*)$
and $Z^* := Z(\lambda^*, y^*)$ such that $Y^* Z^* = 0$. In particular,
we have $\rank(Y^*) + \rank(Z^*) \le n+1$.

\section{Diagonal QCQPs} \label{sec:rankbound}

In this section, we assume (\ref{equ:qcqp}) is a diagonal QCQP, i.e.,
the matrices $C, A_i$ are diagonal. For any fixed index $1 \le j \le n$,
consider the feasibility system
\begin{equation} \label{equ:dfeasj}
y \le 0, \ \ S(y) \succeq 0, \ \ S(y)_{jj} = 0, \ \ s(y)_j = 0.
\end{equation}
Because $S(y)$ is diagonal, this is in fact a polyhedral system, which
by Farkas' Lemma is feasible if and only if the polyhedral system
\begin{align}
    & C \bullet X + c_j x_j = -1 \label{equ:pfeasj} \\
    & A_i \bullet X + a_{ij} x_j \le 0 \ \ \forall \ i = 1, \ldots, m \nonumber \\
    & X \text{ diagonal}, \ X_{kk} \ge 0 \ \ \forall \ k \ne j \nonumber \\
    & X_{jj} \text{ free}, \ x_j \text{ free} \nonumber
\end{align}
is infeasible. It turns out that systems
(\ref{equ:dfeasj})--(\ref{equ:pfeasj}) are key to understanding the
possible ranks of dual feasible $Z(\lambda, y)$. Define
\[
    f := | \{ j : (\ref{equ:dfeasj}) \text{ is infeasible} \} |
       = | \{ j : (\ref{equ:pfeasj}) \text{ is feasible} \} |.
\]
We call $f$ the {\em feasibility number\/} for (\ref{equ:qcqp}),
although it is important to note that $f$ does not depend on the right-hand side $b$.

\begin{lemma} \label{lem:main}
For any dual feasible $\lambda$, $y$, and $Z := Z(\lambda, y)$, it holds
that $\rank(Z) \ge f$.
\end{lemma}

\begin{proof}
Define $S := S(y)$ and $s := s(y)$. We first note that $\rank(Z) \ge
\rank(S)$ since $S$ is a principal submatrix of $Z$. If $\lambda =
0$, then $Z \succeq 0$ implies $s = 0$, which in turn implies that
at least $f$ entries of $\diag(S)$ are positive. Hence, $\rank(Z)
\ge \rank(S) \ge f$. If $\lambda > 0$, then the Schur complement $S
- \lambda^{-1} ss^T$ is positive semidefinite; in particular, $s_j =
0$ whenever $S_{jj} = 0$. Hence, the number of positive entries of
$\diag(S)$ is at least $f$, and $\rank(Z) \ge \rank(S) \ge f$.
\end{proof}

Using Lemma \ref{lem:main}, we prove our main result in this section,
which bounds the rank of any optimal $Y^*$ of (\ref{equ:psdp}).

\begin{theorem} \label{the:main}
Let $Y^* := Y(x^*, X^*)$ be any optimal solution of (\ref{equ:psdp}).
It holds that $1 \le \rank(Y^*) \le n - f + 1$.
\end{theorem}

\begin{proof}
As discussed above, $\rank(Y^*) + \rank(Z^*) \le n+1$, where $Z^*
:= Z(\lambda^*, y^*)$ is optimal for (\ref{equ:dsdp}). Lemma
\ref{lem:main} guarantees $\rank(Z^*) \ge f$, which implies $\rank(Y^*)
\le n - f + 1$. Also, since $Y^*$ is nonzero due to its top-left entry,
$\rank(Y^*) \ge 1$.
\end{proof}

As mentioned in the Introduction, the Laurent-Varvitsiotis rank bound
is $r^* \le 2$, while Theorem \ref{the:main} ensures $r^* \le n - f +
1$. Sections \ref{sec:homog} and \ref{sec:CI0} below give classes of
examples for which $n - f + 1 = 1 < 2$, i.e., $f = n$ and our bound is
tighter than the Laurent-Varvitsiotis bound, but here we would briefly
like to give an example for which our bound is worse. Consider the
standard binary knapsack problem
\[
    \min \left\{ c^T x : a_1^T x \le b_1, x \in \{0,1\}^n \right\},
\]
where every $c_j < 0$ (since the standard knapsack maximizes with
positive objective coefficients) and every $a_{1j} > 0$. In this case,
using the fact that $x_j \in \{0,1\}$ if and only if $x_j = x_j^2$, the
$j$-th system (\ref{equ:pfeasj}) is
\[
    c_j x_j = -1, \ \ \ 
    a_{1j} x_j \le 0, \ \ \
    X_{kk} = 0 \ \ \forall \ k \ne j, \ \ \
    X_{jj} - x_j = 0
\]
which is clearly infeasible since $c_j x_j = -1$ implies $x_j > 0$,
while $a_{1j} x_j \le 0$ implies $x_j \le 0$. Hence, in this example, $f
= 0$, and our bound is $r^* \le n+1$.

An alternative approach to provide a sufficient condition is to check
the feasibility system
\begin{equation} \label{equ:dfeasmaxj}
y \le 0, \ \ S(y) \succeq 0, \ \ S(y)_{jj} = s(y)_j = 0 \ \forall \ j \in J
\end{equation}
for a fixed index set $J\subset \{1,...,n\}$. Again, because $S(y)$ is diagonal, this is in fact a polyhedral system. Then we have
\begin{corollary}
Let $Y^* := Y(x^*, X^*)$ be any optimal solution of (\ref{equ:psdp}).
It holds that $1\le \rank(Y^*) \le j^*$ where $j^*$ is the smallest-cardinality such that all 
systems (\ref{equ:dfeasmaxj}) with $|J| = j^*$ are infeasible.
\end{corollary}
\noindent For example, suppose (\ref{equ:dfeasmaxj}) is infeasible for
all $J$ of size 2. This means that $S(y)$ must have $n-1$ positive
entries, in which case $\rank(Z(y)) \ge \rank(S(y)) \ge n-1$, in
which case $\rank(Y^*) \le n + 1 - (n-1) = 2$. This condition could
be stronger than the bound given by Theorem \ref{the:main} (since the
quantities of multiple indices need to be $0$ at the same time), but it
needs to solve a larger collection of linear programs.

\subsection{The convex case and a perturbation} \label{sec:convex}

As a first application of Theorem \ref{the:main}, we reprove the
classical result---for the case of diagonal QCQPs---mentioned in the
Introduction that the minimum rank $r^*$ equals 1 when (\ref{equ:qcqp})
is a convex program. Of course, Proposition \ref{pro:convex} below holds even
when $C, A_i$ are general positive semidefinite matrices, not just
diagonal ones (see \cite{Fujie-Kojima_1997} for example), but the theory
of this section only applies directly to the diagonal case. (Section
\ref{sec:general} will generalize this result further.)

\begin{proposition} \label{pro:convex}
In the diagonal-QCQP case, suppose (\ref{equ:qcqp}) satisfies $C \succeq 0$ and
$A_i \succeq 0$ for all $i = 1, \ldots, m$. Then there exists an optimal solution
$Y^* := Y(x^*, X^*)$ of (\ref{equ:psdp}) with $\rank(Y^*) = 1$.
\end{proposition}

\begin{proof}
Let us first consider the subcase $C \succ 0$, i.e., $C_{jj} > 0$ for
all $j$. Each of the $n$ linear systems (\ref{equ:pfeasj}) has the
form
\begin{align*}
    &C \bullet X + c_j x_j = -1 \\
    &A_i \bullet X + a_{ij} x_j \le 0 \ \ \forall \ i = 1, \ldots, m
\end{align*}
where $X_{jj}$ and $x_j$ are free, while the remaining variables in
the diagonal $X$ are nonnegative. By setting $x_j = 0$ and $X_{kk}
= 0$ for all $k \ne j$, the system reduces to $C_{jj} X_{jj} = -1$
and $[A_i]_{jj} X_{jj} \le 0$ for all $i$. Then taking $X_{jj} =
-C_{jj}^{-1} < 0$ and using the fact that every $[A_i]_{jj} \ge 0$, we
see that each system is feasible. It follows that $f = n$, and so $r^* =
1$ by Theorem \ref{the:main}.

Now consider the case when some $C_{jj} = 0$. Perturbing $C$ to
$C + D$, where $D$ is a small, positive diagonal matrix, we can
apply the previous paragraph to prove that the SDP relaxation of the
perturbed problem has $r^* = 1$. Now, to complete the proof, we let
$D \to 0$. Note that the perturbation only affects the
objective, and hence we obtain a sequence $\{ Y^* \}$ of rank-1
matrices, each of which is feasible for (\ref{equ:psdp}) and optimal for
its corresponding perturbed SDP. The sequence is also bounded because
the feasible set of (\ref{equ:psdp}) is bounded. Thus, there exists a limit point $\bar
Y$, which is optimal for (\ref{equ:psdp}) and has rank 1. This proves
$r^* = 1$ as desired.
\end{proof}

The proof of Proposition \ref{pro:convex} relies on a perturbation idea
that we will use several times below. The basic insight is that the
feasibility number $f$ can increase under slight perturbations of the
data of (\ref{equ:qcqp}), which means that a nearby SDP relaxation might
enjoy a smaller rank. By letting the perturbation go to 0, we can ensure
that the original SDP contains at least one optimal solution with rank
smaller than could otherwise be guaranteed by a direct application of
Theorem \ref{the:main}.

\subsection{Sign-Definite Linear Terms} \label{sec:homog}

We next reprove a result of \cite{Sojoudi-Lavaei_2014}, tailored to our
diagonal case, that $r^* = 1$ when, for every $j$, the coefficients
$c_j, a_{1j}, \ldots, a_{mj}$ are all nonnegative or all
nonpositive. In such a case, the coefficients are said to be {\em
sign-definite}.

% The objective of (\ref{equ:qcqp}) is {\em homogeneous\/} when $c = 0$,
% and similarly we say that the $i$-th constraint is homogeneous when
% $a_i = 0$. When $c = a_1 = \cdots = a_m = 0$, problem (\ref{equ:qcqp})
% can be solved in polynomial-time as an LP by replacing each squared
% term $x_j^2$ with a variable $y_j \ge 0$. The following proposition
% establishes that, if all constraints are homogeneous but the objective
% is not, then we can still solve (\ref{equ:qcqp}) in polynomial-time
% because the SDP relaxation is exact. We need the following lemma:

\begin{lemma} \label{lem:nonzerocjzeroaij}
Given $1 \le j \le n$, suppose $c_j \ne 0$ and $a_{1j}, \ldots, a_{mj}$
are sign-definite. Then (\ref{equ:pfeasj}) is feasible.
\end{lemma}

\begin{proof}
Take $X = 0$ and $x_j = -c_j^{-1}$. Then the equation $C \bullet X + c_j
x_j = -1$ is satisfied, and the inequalities $A_i \bullet X + a_{ij} x_j
\le 0$ are satisfied becaues $a_{ij}$ and $x_j$ have opposite signs.
\end{proof}

\begin{proposition}[see also \cite{Sojoudi-Lavaei_2014}] \label{pro:homogeneous}
In the diagonal-QCQP case, suppose (\ref{equ:qcqp}) has the property that, for all $j =
1,\ldots,n$, $c_j$ and $a_{ij}$ for all $i = 1, \ldots, m$ are
sign-definite. Then there exists an optimal solution $Y^* := Y(x^*,
X^*)$ of (\ref{equ:psdp}) with $\rank(Y^*) = 1$.
\end{proposition}

\begin{proof}
We consider two subcases. First, when $c_j \ne 0$ for all $j =
1, \ldots, n$, by Lemma \ref{lem:nonzerocjzeroaij} and Theorem
\ref{the:main}, we have $\rank(Y^*) = 1$. When some $c_j = 0$, choose a
fixed $d \in \RR^n$ such that $d_j \ne 0$ for all $j$ with $c_j =
0$, $d_j = 0$ otherwise, and the sign-definite property is maintained.
Also choose $\epsilon > 0$ and perturb $c$ to $c + \epsilon d$, which
in particular does not change the feasible set of (\ref{equ:qcqp}). By
the previous case, the perturbed SDP relaxation has a rank-1 optimal
solution. By letting $\epsilon \to 0$ and using an argument similar
to the proof of Proposition \ref{pro:convex}, we conclude that the
unperturbed (\ref{equ:psdp}) also has a rank-1 optimal solution.
\end{proof}

The diagonal assumption in Proposition \ref{pro:homogeneous}
is necessary because Burer and Anstreicher \cite{Burer-Anstreicher_2013}
provide an example in which $m = 2$, $C$ is non-diagonal, $A_1, A_2$
are diagonal, $c \ne 0$, $a_1 = a_2 = 0$, and the Shor relaxation is
not tight; in particular, it has no optimal solution with rank 
1. On the other hand, the diagonal assumption can at least be relaxed when $m =
2$ for the purely homogeneous case: Ye and Zhang \cite{Ye.Zhang.2003a}
showed that, if $m = 2$ with $C, A_1, A_2$ arbitrary and $c = a_1 =
a_2 = 0$, then the corresponding Shor relaxation has a rank-1 optimal
solution.

An interesting application of Proposition \ref{pro:homogeneous} occurs
for the feasible set
\begin{equation} \label{equ:balls}
    \left\{ x \ : \ \| x \|_2 \le \rho_1, \ \| x \|_\infty \le \rho_2 \right\}
    = \left\{ x \ : \ x^T x \le \rho_1^2, \ x_j^2 \le \rho_2^2 \ \forall \ j \right\}
\end{equation}
which is the intersection of concentric $2$-norm and $\infty$-norm
balls. It is well known that, for only the 2-norm ball $\{x : x^T x \le \rho_1^2
\}$, problem (\ref{equ:qcqp}) is equivalent to the trust-region
subproblem, which can be solved in polynomial time. On the other hand,
for only the $\infty$-norm ball $\{ x : x_j^2 \le \rho_2^2 \ \forall \ j \}$, problem (\ref{equ:qcqp}) is clearly separable and hence solvable
in polynomial time. Proposition \ref{pro:homogeneous} shows
that (\ref{equ:qcqp}) over the intersection (\ref{equ:balls}) can also
be solved in polynomial-time.

According to theorem 2 of \cite{Yang-et-al_2017}, the fact that
(\ref{equ:psdp}) solves (\ref{equ:qcqp}) when the sign-definiteness
property holds also allows us to relate the feasible set of
(\ref{equ:psdp}) to the closed convex hull
\[
    {\cal K} := \overline{\text{conv}}\left\{ (x, x \circ x) :
    x \text{ feasible for (\ref{equ:qcqp})} \right\}
\]
where $\circ$ denotes the Hadamard, i.e., component-wise, product of
vectors. Such convex hulls are important for studying QCQPs in general.
Specifically, we know that
\[
    {\cal K} = \{ (x, \diag(X)) : (x,X) \text{ is feasible for (\ref{equ:psdp})} \}
\]
when sign-definiteness holds. In this sense, problem (\ref{equ:qcqp}) is
a ``hidden convex'' problem in this case.

\subsection{Arbitrary $C$ and each $A_i \in \{\pm I, 0\}$} \label{sec:CI0}

As mentioned above, Proposition \ref{pro:homogeneous} of the previous
subsection was first proved in \cite{Sojoudi-Lavaei_2014}, and
it involves only conditions on the data $c_j$ and $a_{ij}$ of
the linear terms in (\ref{equ:qcqp}). In fact, the authors of
\cite{Sojoudi-Lavaei_2014} provide a broader theory, one that studies
more general nonzero structures---not just diagonal---but one that
only considers data corresponding to off-diagonal terms $X_{ij}$ in
the SDP relaxations. In particular, they do not consider data such as
$C_{jj}$ and $[A_i]_{jj}$. This is indeed a key difference of our theory
compared to theirs, i.e., our feasibility number $f$ takes into account
the data matrices $C, A_i$. We now give an example to illustrate this
point further---an example in which the sign-definiteness assumption in
Proposition \ref{pro:homogeneous} can be relaxed when $C, A_i$ are taken
into account.

As discussed in the Introduction, the assumption that the matrices
$C, A_1, \ldots, A_m$ are diagonal is equivalent (after a linear
transformation) to the matrices pairwise commuting. When $C$ is
arbitrary and $A_i \in \{\pm I, 0\}$ for all $i$, this assumption
is clearly satisfied. Geometrically, the feasible set is then an
intersection of balls, complements of balls, and half-spaces.
Although this problem is strongly NP-hard in general, Bienstock and
Michalka \cite{Bienstock-Michalka_2014} show that it can be solved in
polynomial-time, for example, when the number of ball constraints is
fixed. More recently, Beck and Pan \cite{Beck-Pan_2017} study precisely
this special case of (\ref{equ:qcqp}) and develop a branch-and-bound
algorithm for its global optimization; \cite{Beck-Pan_2017} also
contains a detailed literature review of this problem.

Assume that the data has already been transformed so that $C$ is
diagonal and, without loss of generality, $C_{11} \ge \cdots \ge
C_{nn}$. In particular, the diagonal of $C$ contains the eigenvalues
of the original $C$. In addition, let us consider the sub-case
in which $c_n$ and $a_{in}$ for all $i$ are sign-definite. (This
is a weaker condition than the sign-definiteness of Proposition
\ref{pro:homogeneous}.) By Theorem \ref{the:main}, the rank of
an optimal solution $Y^*$ of the corresponding Shor relaxation
(\ref{equ:psdp}) is bounded above by $n - f + 1$, where $f$ is the
feasibility number associated with the systems
\begin{align*}
    C \bullet X + c_j x_j &= -1 \\
    \pm I \bullet X + a_{ij} x_j &\le 0 \ \ \ \ \forall \ i \text{ with } A_i = \pm I \\
    a_{ij} x_j &\le 0 \ \ \ \ \forall \ i \text{ with } A_i = 0
\end{align*}
where $X$ is diagonal, $X_{jj}$ and $x_j$ are free, while the remaining
variables $X_{kk}$ are nonnegative. Our next proposition shows that
$f$ equals $n$, so that $r^* = 1$.

\begin{proposition} \label{pro:CI0}
If $C$ is diagonal with $C_{11} \ge \cdots \ge C_{nn}$, $A_i \in \{\pm
I, 0\}$ for all $i = 1,\ldots,m$, and $c_n, a_{1n},\ldots,a_{mn}$
sign-definite, then $r^* = 1$.
\end{proposition}

\begin{proof}
We first examine the sub-case when $C_{(n-1)(n-1)} > C_{nn}$ and $c_n \ne 0$.
For $j = 1, \ldots, n-1$, consider the system described above the statement
of the proposition. Fixing $x_j = 0$, it reduces
to the system $C \bullet X = -1$, $\pm I \bullet
X \le 0$. Next fixing $X_{jj} = -\sum_{k \ne j} X_{kk}$, the
system further simplifies to
\[
    \sum_{k \ne j} (C_{kk} - C_{jj}) X_{kk} = -1.
\]
We may then take $X_{nn} = (C_{jj} - C_{nn})^{-1}$, which is postive
since $C_{jj} > C_{nn}$, and all other $X_{kk}
= 0$, showing that the system is feasible. Now consider the system for
$j = n$ above. Setting $X_{nn} = -\sum_{k = 1}^{n-1} X_{kk}$, the
system reduces to $c_n x_n = -1$ and $a_{in} x_n \le 0$ for all $i$.
Because $c_n$ and $a_{in}$ are sign-definite and because $c_n \ne 0$, this system is solvable.
It follows that $f = n$ when $C_{(n-1)(n-1)} > C_{nn}$ and $c_n \ne 0$.

Finally, if $C_{(n-1)(n-1)} = C_{nn}$ or $c_n = 0$, then we may make an
arbitrarily small perturbation of the objective such that the previous
paragraph applies. As in the proof of Proposition \ref{pro:convex}, the
perturbation can be removed, thus establishing $r^* = 1$.
\end{proof}

\section{General QCQPs} \label{sec:general}

We now turn our attention to the case of general, non-diagonal QCQPs,
keeping in mind that Assumptions \ref{ass:feas}--\ref{ass:int} still
apply. In particular, the feasible set of (\ref{equ:qcqp}) is bounded,
i.e., it exists in a ball defined by $x^T x \le \rho^2$ for some radius
$\rho$. Note that, for the following development, $\rho$ does not need to
be known explicitly.

We do assume, for simplicity and without loss of generality, that $C$
is diagonal, and we let $A_i = Q_i D_i Q_i^T$ denote the spectral
decomposition of $A_i$, where $Q_i$ is an orthogonal matrix. Next,
we introduce auxiliary variables $y_i = Q_i^T x \in \RR^n$, rewriting
(\ref{equ:qcqp}) as
\begin{align}
\min \ \ & x^T C x + 2 c^T x \label{equ:qcqp'} \\
\st \ \ & y_i^T D_i y_i + 2 a_i^T x \le b_i, \ y_i = Q_i^T x \nonumber \\
& x^T x + \sum_i y_i^T y_i \le (m + 1) \rho^2 \nonumber
\end{align}
where the last constraint is technically redundant but has
been added so that (\ref{equ:qcqp'}) more clearly satisfies Assumptions
\ref{ass:feas}--\ref{ass:int} on its own. In particular, the feasible
set of (\ref{equ:qcqp'}) is bounded.

The lifted problem (\ref{equ:qcqp'}) is a diagonal QCQP, and so the
Laurent-Varvitsiotis bound \cite{Laurent-Varvitsiotis_2014} guarantees
$r^* \le 2$ for the Shor SDP relaxation of (\ref{equ:qcqp'}). As
mentioned in the Introduction, Madani et al.~\cite{Madani-et-al_2014}
have previously shown that every polynomial optimization problem can be
reformulated as a polynomial-sized QCQP, which has an SDP relaxation
with $r^* \le 2$. Because every polynomial optimization problem can
be mechanically converted to a QCQP, which can then be converted to a
diagonal QCQP as above, the Laurent-Varvitsiotis bound of $r^* \le 2$ is
an alternate---and in our opinion, simplified---derivation of the same
result. Interestingly, these results show that boundary between ``easy''
SDP relaxations and ``hard'' polynomial optimization problems lies
between $r^* = 2$ and $r^* = 1$.

We can also apply the theory of Section \ref{sec:rankbound} to (\ref{equ:qcqp'}) since
it is a diagonal QCQP. In particular, we would
like to determine sufficient conditions under which the feasibility
number for (\ref{equ:qcqp'}) equals its total number of variables, which
is $n(m+1)$. We provide just such a condition in the following theorem,
which is an analog of Theorem \ref{the:main}. To this end, we introduce
the following linear system:
\begin{align}
    & C \bullet X = -1 \label{equ:pfeasj'} \\
    & D_i \bullet Y_i \le 0 \ \forall \ i = 1,\ldots,m \nonumber \\
    & I \bullet X + \textstyle{\sum_{i = 1}^m} I \bullet Y_i \le 0 \nonumber \\
    & X, Y_i \text{ diagonal}. \nonumber
\end{align}

\begin{theorem} \label{the:general}
Let $z$ represent any single variable $X_{jj}$ or $[Y_i]_{jj}$ in
(\ref{equ:pfeasj'}). Constrain (\ref{equ:pfeasj'}) further by forcing
all variables other than $z$ to be nonnegative, while keeping $z$ free.
If all such $n(m+1)$ systems corresponding to every possible choice
of $z$ are feasible, then $r^* = 1$ for both (\ref{equ:qcqp'}) and
(\ref{equ:qcqp}).
\end{theorem}

\begin{proof}
The $n(m+1)$ systems (\ref{equ:pfeasj'}) constitute the systems
(\ref{equ:pfeasj}) tailored to (\ref{equ:qcqp'}), reduced further by
setting the ``linear part'' $x_j$ in (\ref{equ:pfeasj}) to 0. So we
conclude that $r^* = 1$ for (\ref{equ:qcqp'}) by applying Theorem
\ref{the:main}. The result $r^* = 1$ also holds for (\ref{equ:qcqp})
because the SDP relaxation (\ref{equ:psdp}) for (\ref{equ:qcqp}) is at
least as strong as the corresponding relaxation for (\ref{equ:qcqp'}),
which we have just proven is exact.
\end{proof}

\noindent Note that the feasibility of the systems (\ref{equ:qcqp'}) can
be checked in polynomial time.

As mentioned in the above proof, system (\ref{equ:pfeasj'}) is a direct
application of (\ref{equ:pfeasj}) to problem (\ref{equ:qcqp'}) with the
following additional restriction: relative to (\ref{equ:pfeasj}), the
term $x_j$ is fixed at 0. Said differently, system (\ref{equ:pfeasj'})
does not include the effects of the linear terms of (\ref{equ:qcqp'}),
e.g., the terms $a_i^T x$, $y$, and $Q_i^T x$. While this may
seem like a major restriction, we will see next---and in Section
\ref{sec:random}---that (\ref{equ:pfeasj'}) retains enough flexibility
to prove that the feasibility number of (\ref{equ:qcqp'}) is indeed $n(m
+ 1)$ for some interesting cases. The key to retaining this flexibility
is actually a consequence of the redundant constraint $x^T x + \sum_i
y_i^T y_i \le (m + 1) \rho^2$ in (\ref{equ:qcqp'}) and its counterpart
in (\ref{equ:pfeasj'}).

Similar to Theorem \ref{the:main}, perturbation can be a
useful tool for broadening the application of Theorem \ref{the:general}
by making feasibility systems like (\ref{equ:pfeasj'}) more likely to be
feasible. For example, a reasonable perturbation might be to replace the
objective of $x^T C x + 2 c^T x$ of (\ref{equ:qcqp'}) with $x^T C x + 2
c^T x + \epsilon \sum_{i=1}^m y_i^T y_i$, where $\epsilon > 0$ is small,
resulting in the analog
\begin{equation} \label{equ:pfeasj''}
    C \bullet X + \epsilon \sum_{i=1}^m I \bullet Y = -1, \ \ \
    D_i \bullet Y_i \le 0, \ \ \
    I \bullet X + \sum_{i = 1}^m I \bullet Y_i \le 0
\end{equation}
of (\ref{equ:pfeasj'}). Note that this particular perturbation is consistent with
the need to satisfy the inequality $I \bullet X + \sum_{i = 1}^m I
\bullet Y_i \le 0$. The following proposition, which proves the general
convex case of (\ref{equ:qcqp}), is an example of this perturbation.

\begin{proposition}[see \cite{Fujie-Kojima_1997}]
Suppose (\ref{equ:qcqp}) satisfies $C \succeq 0$ and $A_i \succeq 0$ for
all $i = 1, \ldots, m$. Then there exists an optimal solution $Y^* :=
Y(x^*, X^*)$ of (\ref{equ:psdp}) with $\rank(Y^*) = 1$.
\end{proposition}

\begin{proof}
First assume $C \succ 0$. Using the suggested perturbation, we need
to show all such systems (\ref{equ:pfeasj''}) are feasible. For the
system with $X_{jj}$ free, set all other variables to 0 so that
(\ref{equ:pfeasj''}) reduces to $C_{jj} X_{jj} = -1$ and $X_{jj} \le
0$, which is solvable because $C_{jj} > 0$. On the other hand, for the
systems with $[Y_i]_{jj}$ free, set all other variables to 0. Then
(\ref{equ:pfeasj''}) becomes $\epsilon [Y_i]_{jj} = -1$, $[D_i]_{jj}
[Y_i]_{jj} \le 0$, and $[Y_i]_{jj} \le 0$, which is also solvable
because $[D_i]_{jj} \ge 0$. Hence the feasibility number is $n(m+1)$ as
desired. The case $C \succeq 0$ is just a limiting case of $C \succ 0$.
\end{proof}

\section{Random General QCQPs} \label{sec:random}

In this section, we study the behavior of $r^*$ for (\ref{equ:qcqp})
under the assumption that $C$ is positive semidefinite and the $A_i$
are generated randomly. The analysis is an extension of the ideas
of Section \ref{sec:general}, and it does not depend on $c, a_i$,
or $b$, although these data are required for satisfying Assumptions
\ref{ass:feas}--\ref{ass:int}. Our result is as follows:

\begin{theorem} \label{the:random}
Regarding the general QCQP (\ref{equ:qcqp}), suppose that $C$ is
positive semidefinite, and for each $i = 1,\ldots,m$, $A_i$ is generated
randomly with eigenvalues independently following the standard
Gaussian distribution. Suppose also that $c, a_1, \ldots, a_m$, and
$b$ are chosen independently such that Assumptions \ref{ass:feas} and
\ref{ass:int} are satisfied. Finally, for any finite $\rho > 0$, add the
constraint $x^T x \le \rho^2$ to ensure that Assumption \ref{ass:pd}
is satisfied, while not violating Assumptions \ref{ass:feas} and
\ref{ass:int}. Then, if $m \le n^q$ for a fixed positive integer $q$, $\text{Prob}(r^* =
1) \to 1$ as $n \to \infty$.
\end{theorem}

\noindent The proof will make use of the following lemma:

\begin{lemma}
Let $\beta \in (0,1)$, and let $p, q$ be positive integers. Then
$
    \lim_{p \to \infty} p^q \log(1 - \beta^p) = 0,
$
where $\log$ is the natural logarithm.
\end{lemma}

\begin{proof}
Consider the change of variables $x = 1/p$ so that the limit becomes
\[
    \lim_{x \to 0^+} \frac{\log(1 - \beta^{1/x})}{(1/x)^q} = 0,
\]
which, by l'H\^{o}pital's rule, equals
\[
    \lim_{x \to 0^+} \frac{-\log(\beta) \beta^{1/x} x^{q-1}}{q(1 - \beta^{1/x})} = 0.
\]
\end{proof}

\begin{proof}[Proof of Theorem \ref{the:random}]
We analyze the situation when $C \succ 0$, as $C \succeq 0$ is just a
limiting case. Without loss of generality, after a change of variables,
we may assume that $C$ is diagonal. Following the development in Section
\ref{sec:general}, our randomly generated problem (\ref{equ:qcqp}) with
the added constraint $x^T x \le \rho^2$ is equivalent to (\ref{equ:qcqp'}).
We claim that $r^* = 1$ for (\ref{equ:qcqp'}) with high probability,
and since the SDP relaxation for (\ref{equ:qcqp'}) is at least as tight
as (\ref{equ:psdp}) for (\ref{equ:qcqp}), this will prove $r^* = 1$ for
(\ref{equ:qcqp}) with high probability as desired.

To prove the claim, we analyze a perturbation of (\ref{equ:qcqp'}).
For each $i$, let $B_i \in \SSS^n$ be a diagonal matrix with diagonal
entries independently following the uniform distrubtion on $[0,1]$, and
for $\epsilon > 0$ small, consider the perturbed problem
\begin{align}
\min \ \ & x^T C x + 2 c^T x + \epsilon \textstyle{\sum_i} y_i^T B_i y_i \label{equ:qcqp''} \\
\st \ \ & y_i^T D_i y_i + 2 a_i^T x \le b_i, \ y_i = Q_i^T x \nonumber \\
& x^T x + \sum_i y_i^T y_i \le (m + 1)\rho^2. \nonumber
\end{align}
Analogous to systems (\ref{equ:pfeasj''}) in Section \ref{sec:general},
we analyze the feasibility of systems of the
form
\begin{align}
& C \bullet X + \epsilon \textstyle{\sum_i} B_i \bullet Y_i = -1 \label{equ:local1} \\
& D_i \bullet Y_i \le 0 \ \forall \ i \nonumber \\
& I \bullet X + \textstyle{\sum_i} I \bullet Y_i \le 0 \nonumber
\end{align}
where all matrices $X$, $Y_i$ are diagonal and all variables are
nonnegative except for one, which is free.

First consider the case of (\ref{equ:local1}) when $X_{jj}$ is free. Set
all $Y_i = 0$ so that the system reduces to $C \bullet X = -1$ and $I
\bullet X \le 0$, which is feasible since $C_{jj} > 0$ by assumption.

Second, consider the case of (\ref{equ:local1}) when $[Y_k]_{jj}$ is
free. Set $X = 0$ and all other $Y_i = 0$ so that the system reduces
to $\epsilon B_k \bullet Y_k = -1$, $D_k \bullet Y_k \le 0$, and $I
\bullet Y_k \le 0$, which is certainly feasible if the following
equality system is feasible:
\begin{align}
& B_k \bullet Y_k = -1 \label{equ:local2} \\
& D_k \bullet Y_k = 0 \nonumber \\
& I \bullet Y_k = 0. \nonumber
\end{align}
Note that (\ref{equ:local2}) does not depend on $\epsilon$. The basis
size for (\ref{equ:local2}) is 3, and due to the random nature of
the data, every $3 \times 3$ basis matrix is invertible. Also, because
$[Y_k]_{jj}$ is free while all other variables in $Y_k$ are nonnegative,
the system (\ref{equ:local2}) has $n -1 \choose 2$ bases, and hence,
$n-1 \choose 2$ basic solutions.

Let $\alpha$ be the probability of any particular basic solution of
(\ref{equ:local2}) being feasible. In particular, $\alpha$ is the same
for every basic solution since the entries of $B_k$ are generated
independently and identically and similarly for the entries of $D_k$.
Hence, $\alpha$ is equal to the probability that the solution $(\bar
y_1, \bar y_2, \bar y_3)$ of the random system
\begin{align*}
    b_1 y_1 + b_2 y_2 + b_3 y_3 &= -1 \\
    d_1 y_1 + d_2 y_2 + d_3 y_3 &= 0 \\
    y_1 + y_2 + y_3 &= 0
\end{align*}
satisfies $\bar y_2 \ge 0$ and $\bar y_3 \ge 0$, where $b_1, b_2, b_3$
are i.i.d.~uniform in $[0,1]$ and $d_1, d_2, d_3$ are i.i.d.~standard
normal. Observing that the specific realizations $(0.5, 0.5, 0.4)$ and
$(0, -1, 1)$ of $b$ and $d$, respectively, yield $(\bar y_1, \bar y_2,
\bar y_3) = (-20, 10, 10)$, we conclude that there is an open set of
realizations $(b, d)$ satisfying $\bar y_1 > 0$ and $\bar y_2 > 0$.
Hence, the probability of a realization occurring in this open set is
positive, which in turn ensures $\alpha > 0$. Note also that $\alpha$
is independent of $n$ and $m$. (Indeed, empirically we can verify using
Monte Carlo simulation that $\alpha \approx 1/6$.)

Next, due to independence, the probability that (\ref{equ:local2}) is
feasible, i.e., there exists at least one basic feasible solution, is
\[
    \theta := 1 - \left(1 - \alpha\right)^{n - 1 \choose 2}.
\]
Thus, $\theta$ is also a lower bound on the probability of the
feasibility of system (\ref{equ:local1}) when $[Y_k]_{jj}$ is free. To
ensure $r^* = 1$ for (\ref{equ:qcqp''}), we need that all such systems
(\ref{equ:local1}) are feasible. Again exploiting independence, this
occurs with probability at least $\theta^{nm}$. We claim that $\lim_{n
\to \infty} \theta^{mn} = 1$, which is certainly true if
\begin{equation} \label{equ:local3}
    0 = \lim_{n \to \infty} m n \log(\theta) = \lim_{n \to \infty} m n
    \log(1 - (1 - \alpha)^{n-1 \choose 2}).
\end{equation}
Define $\beta := 1 - \alpha$. Since $m \le n^q$ and $\beta \in (0,1)$,
we have
\[
0 \le -m n \log(1 - \beta^{n-1 \choose 2}) \le 
-n^{q+1} \log(1 - \beta^n).
\]
The above lemma thus implies
\[
    0 \le \lim_{n \to \infty} -m n \log(1 - \beta^{n-1 \choose 2}) \le 
    \lim_{n \to \infty} -n^{q+1} \log(1 - \beta^n) = 0,
\]
which proves (\ref{equ:local3}).

Finally, since the above probability analysis does not depend on
$\epsilon$, we may take $\epsilon \to 0$ so that the probability
analysis applies as well to problem (\ref{equ:qcqp'}), which proves the
claim, i.e., that $r^* = 1$ for (\ref{equ:qcqp'}) with high probability.
\end{proof}

Although Theorem \ref{the:random} assumes that $C \succeq 0$, we
conjecture that is true even when $C$ is generated randomly in the
same manner as the $A_i$ matrices. The proof seems to break down
for analyzing the relevant feasibility system for $X_{jj}$, which
corresponds to the smallest eigenvalue $C_{jj}$ of $C$, when that
$C_{jj}$ is negative. A possible work-around could be to put the
objective $x^T C x + 2 c^T x$ into the constraints using an auxiliary
variable $t$ via the constraint $x^T C x + 2 c^T x \le t$ and then to
minimize $t$. However, $t$ would need to be bounded in accordance with
Assumptions \ref{ass:feas}--\ref{ass:int} before applying the theory we
have developed.

\end{onehalfspace}

\section*{Acknowledgements}

We are in debt to the anonymous associate editor and two referees, who
suggested many positive improvements to the paper. We would also like to
thank Gang Luo, who pointed out an error in the knapsack example.

\bibliographystyle{abbrv}
\bibliography{paper}

\end{document}